\documentclass[twoside,12pt,leqno]{amsproc}
\usepackage{amssymb,latexsym,enumerate,verbatim,tikz}
\usetikzlibrary{patterns}
\usepackage[pagebackref]{hyperref}
\usepackage{amsrefs}    
\usepackage{etoolbox}   
\hypersetup{citecolor=red, linkcolor=blue, colorlinks=true}
\numberwithin{table}{section}

\theoremstyle{plain}
\newtheorem{theorem}{Theorem}[section]
\newtheorem{lemma}[theorem]{Lemma}
\newtheorem{proposition}[theorem]{Proposition}
\newtheorem{corollary}[theorem]{Corollary}

\theoremstyle{definition} 

\newtheorem{remark}[theorem]{Remark}
\AtEndEnvironment{remark}{\null\hfill$\triangle$}

\renewcommand{\geq}{\geqslant}

\renewcommand{\ge}{\geqslant}
\renewcommand{\le}{\leqslant}

\newcommand{\F}{\mathbb{F}}
\newcommand{\GL}{\mathrm{GL}}
\newcommand{\inv}{\textup{inv}}
\newcommand{\nil}{\textup{nil}}

\newcommand{\rk}{\mathrm{rk}}

\renewcommand{\Vec}[1]{\overline{#1}}

\makeatletter        
\def\@adminfootnotes{%
  \let\@makefnmark\relax  \let\@thefnmark\relax
  \ifx\@empty\@date\else \@footnotetext{\@setdate}\fi
  \ifx\@empty\@subjclass\else \@footnotetext{\@setsubjclass}\fi
  \ifx\@empty\@keywords\else \@footnotetext{\@setkeywords}\fi
  \ifx\@empty\thankses\else \@footnotetext{%
    \def\par{\let\par\@par}\@setthanks}%
  \fi}\makeatother   

\begin{document}

\hyphenation{}

\title[]{On the dimension of twisted centralizer codes}
\author{Adel Alahmadi, S.\,P. Glasby, Cheryl E. Praeger}

\address[Glasby]{
Centre for Mathematics of Symmetry and Computation\\
University of Western Australia\\
35 Stirling Highway\\
Crawley 6009, Australia. Also affiliated with The Faculty of Information
Sciences and Engineering, University of Canberra, ACT 2601, Australia. Email: {\tt Stephen.Glasby@uwa.edu.au; WWW: \href{http://www.maths.uwa.edu.au/~glasby/}{http://www.maths.uwa.edu.au/$\sim$glasby/}}}
\address[Praeger]
{Centre for Mathematics of Symmetry and Computation\\
University of Western Australia\\
35 Stirling Highway\\
Crawley 6009, Australia. Also affiliated with King Abdulaziz University, Jeddah,
 Saudi Arabia. Email: {\tt Cheryl.Praeger@uwa.edu.au;\newline WWW: \href{http://www.maths.uwa.edu.au/~praeger}{http://www.maths.uwa.edu.au/$\sim$praeger}}}
\address[Alahmadi]{Dept. of  Mathematics, King Abdulaziz University, Jeddah, Saudi Arabia.}

\let\thefootnote\relax 
\date{\today}

\begin{abstract}
Given a field $F$, a scalar $\lambda\in F$ and a matrix $A\in F^{n\times n}$,
the \emph{twisted centralizer code}
$C_F(A,\lambda):=\{B\in F^{n\times n}\mid AB-\lambda BA=0\}$ is a linear code
of length $n^2$ over~$F$. When $A$ is cyclic
and $\lambda\ne0$ we prove that
$\dim C_F(A,\lambda)=\deg(\gcd(c_A(t),\lambda^nc_A(\lambda^{-1}t)))$ where
$c_A(t)$ denotes the
characteristic polynomial of $A$. We also show how $C_F(A,\lambda)$
decomposes, and we estimate the probability that $C_F(A,\lambda)$ is
nonzero when $|F|$ is finite. Finally, we prove
$\dim C_F(A,\lambda)\le n^2/2$ for $\lambda\not\in\{0,1\}$
and `almost all' $n\times n$ matrices~$A$ over $F$.
\end{abstract}

\maketitle

\section{Introduction}\label{S:Intro}

Fix a (commutative) field $F$, a scalar $\lambda\in F$, and a
matrix $A\in F^{n\times n}$. The subspace
\[
  C_F(A,\lambda):=\{B\in F^{n\times n}\mid AB-\lambda BA=0\}.
\]
is called a \emph{twisted centralizer code}. 
We are primarily interested in the case when $F=\F_q$ is a finite field,
in which case $C_F(A,\lambda)$ is a linear code. This paper is
motivated by results in~\cite{CC}. We rely heavily on the fact that replacing
$A$ by a conjugate does not change $\dim C_F(A,\lambda)$.
By contrast, the Hamming weight
of a matrix (the number of nonzero entries), is highly sensitive to
change of basis. 
Coding theory applications are considered in~\cite{TCC}.

Let $c_A(t)$ be the characteristic
polynomial of $A$, namely $\det(tI-A)$, and let $m_A(t)$ be its minimal
polynomial. Let $\overline{F}$ be the algebraic closure of $F$.
Let $S(A)$ denote the set of roots of $m_A(t)$ and hence also of $c_A(t)$
in $\overline{F}$, and let $L$ be the subfield of $\overline{F}$ containing
$S(A)$ and $F$, i.e. the splitting (sub)field for $c_A(t)$. Suppose
$c_A(t)=\prod_{\alpha\in S(A)}(t-\alpha)^{m_\alpha}$ for positive integers
$m_\alpha$. View $L^n$ as a right $L[A]$-module (usually as
$n$-dimensional row vectors over $L$).
For $\alpha\in L$
let $K_\alpha$ be the $\alpha$-eigenspace $\{v\in L^n\mid vA=\alpha v\}$, and
set $k_\alpha=\dim(K_\alpha)$.
Note that $K_\alpha\ne\{0\}$ if and only if $\alpha\in S(A)$.
Thus, in particular, $K_0$ is the row null space of $A$ which we
also denote $\textup{RNull}_L(A)$, and $k_0$ is the nullity of~$A$.
The $(t-\alpha)$-primary $L[A]$-submodule of $L^n$ is
$M_\alpha=\{v\in L^n\mid v(A-\alpha I)^n=0\}$.
Observe that $\dim(M_\alpha)=m_\alpha$ and $K_\alpha\le M_\alpha$, so that
$k_\alpha\le m_\alpha$. To stress the dependence on~$A$, we write
$K_{A,\alpha}, k_{A,\alpha}, M_{A,\alpha}, m_{A,\alpha}$. When $\lambda=1$,
we write $C_F(A)$ instead of $C_F(A,1)$. Note that $C_F(A,\lambda)$
is a module for the $F$-algebra $C_F(A)$.

In Section~\ref{S:nilinv} we see how
$C_F(A,\lambda)$ decomposes, and in Section~\ref{S:Parity} we identify
$C_F(A,\lambda)$ with the null space of a (parity check) matrix.
When $\lambda\ne0$, the `$\lambda$-twisted' characteristic polynomial
$\lambda^n c_A(\lambda^{-1}t)$ is closely related to $\dim C_F(A,\lambda)$.
In Section~\ref{S:cyclic} we prove that $\dim C_F(A,\lambda)=
\deg(\gcd(c_A(t),\lambda^n c_A(\lambda^{-1}t)))$ when $A$ is cyclic
and $\lambda\ne0$. If $\lambda\ne 0$ then 
$C_F(A,\lambda)\cong C_F(A_\nil,\lambda)\oplus C_F(A_\inv,\lambda)$
by Theorem~\ref{T:decom} where $A_\nil$ and $A_\inv$
denote the `nilpotent and invertible parts' of $A$.
Theorem~\ref{T:decom2} gives a decomposition of $C_L(A,\lambda)$ related
to the $(t-\alpha)$-primary $L[A]$-submodules of $L^n$
provided  $F$ has prime characteristic, and a certain condition holds
involving
$\lambda^n c_A(\lambda^{-1}t)$ and~$\lambda^{-n} c_A(\lambda t)$.

If $\det(A)=0$, then $C_F(A,\lambda)\ne\{0\}$
for all $\lambda\in F$ by Corollary~\ref{C:det}.
In Section~\ref{S:Quokka} we show  when $\lambda\not\in\{0,1\}$
that the probability is positive, that a uniformly distributed
$A\in\F_q^{n\times n}$ has $C_{\F_q}(A,\lambda)\ne\{0\}$.
In Section~\ref{S:UB} we establish upper bounds for
$\dim C_F(A,\lambda)$, and Theorem~\ref{T:UB} 
proves that $\lambda\not\in\{0,1\}$ and $k_0+m_0/2\le n$
implies that $\dim C_F(A,\lambda)\le n^2/2$.  The bound
$n^2/2$ is attained when $n$ is even and $\lambda=-1$ by Remark~\ref{R:2m}.

\section{The nilpotent and invertible parts of \texorpdfstring{$A$}{}}
\label{S:nilinv}

A matrix $A\in F^{n\times n}$ is called \emph{nilpotent} if
$A^n=0$. 
The $n$-dimensional row space $V=F^n$ decomposes as $V=V_\nil\oplus V_\inv$
where 
\[
  V_\nil:=\{v\in V\mid vA^n=0\}\qquad\textup{and}\qquad
  V_\inv:=\{vA^n\mid v\in V\}.
\]
Thus $A$ is conjugate to a block diagonal matrix $A_\nil\oplus A_\inv$
where $A_\nil$ is nilpotent on $V_\nil$, and $A_\inv$ is invertible on~$V_\inv$.

\begin{lemma}\label{L:k0}
If $A\in F^{n\times n}$ has nullity $k_0$, then
$\dim C_F(A,0)=k_0n$. In particular, if $A\ne0$ then $\dim C_F(A,0)\le n(n-1)$,
and $C_F(A,0)$ contains no invertible matrices.
\end{lemma}

\begin{proof}
A necessary and sufficient condition for $B\in C_F(A,0)$ is $AB=0$, or that
every column of $B$ lies in the column null space of $A$. Thus $A\ne0$ implies
that each such $B$ is singular, and since
$\dim\textup{CNull}_F(A)=\dim\textup{RNull}_F(A)$ equals $k_0$, we have
$\dim C_F(A,0)=k_0n$. The maximum value of $k_0$ is $n-1$ when $A\ne0$. 
\end{proof}

Let $E_{i,j}\in F^{n\times n}$ have 1 in position $(i,j)$ and zeros elsewhere.
If $A$ is conjugate to $E_{1,1}$ or $E_{1,2}$, then $\dim C_F(A,0)=(n-1)n$
by Lemma~\ref{L:k0}. We see later (Corollary~\ref{C:bound})
that, for any $\lambda$, $d=\dim C_F(A,\lambda)$ cannot lie
in the range $(n-1)n<d<n^2$.  For $F$-subspaces of matrices $U$ and $W$, we write 
$U\cong W$ to mean that $U$ and $W$ are isomorphic as vector spaces.   

\begin{theorem}\label{T:decom}
If $A\in F^{n\times n}$ and $0\kern-1pt\ne\kern-1pt\lambda\in F$, then
$C_F(A,\lambda)\cong C_F(A_\nil,\lambda)\oplus C_F(A_\inv,\lambda)$.
\end{theorem}

\begin{proof}
By replacing $A$ by a conjugate of itself, we may assume that
$A=\left(\begin{smallmatrix}A_1&0\\0&A_2\end{smallmatrix}\right)$
where $A_1:=A_\nil$ and $A_2:=A_\inv$.
Now $B=\left(\begin{smallmatrix}B_1&B_2\\B_3&B_4\end{smallmatrix}\right)$
lies in $C_F(A,\lambda)$ precisely when
\begin{equation}\label{E:}
  AB-\lambda BA
  =\begin{pmatrix}A_1B_1-\lambda B_1A_1&A_1B_2-\lambda B_2A_2\\
  A_2B_3-\lambda B_3A_1&A_2B_4-\lambda B_4A_2\end{pmatrix}
  =\begin{pmatrix}0&0\\0&0\end{pmatrix}.
\end{equation}
Equating blocks give the following four conditions: $B_1\in C_F(A_1,\lambda)$,
$B_4\in C_F(A_2,\lambda)$, 
\[
  B_2=\lambda^{-1}A_1B_2A_2^{-1}\quad\textup{and}\quad
  B_3=\lambda A_2^{-1}B_3A_1.
\]
Iterating the equation $B_2=\lambda^{-1}A_1B_2A_2^{-1}$ gives
$B_2=\lambda^{-n}A_1^nB_2A_2^{-n}$. Since $A_1^n=0$, we see that $B_2=0$.
A similar calculation shows that $B_3=0$.
Therefore $(B_1,B_4)\mapsto\textup{diag}(B_1,B_4)$ defines a (distance-preserving) isomorphism from
$C_F(A_1,\lambda)\oplus C_F(A_2,\lambda)$ to $C_F(A,\lambda)$.
\end{proof}

\begin{remark}
Adapting the proof of Theorem~\ref{T:decom}
shows that $C_F(A,0)$ is independent of $A_\inv$;
more precisely $C_F(A,0)\cong \{X\in F^{m_0\times n}\mid A_\nil X=0\}$
where $A$ has order $m_0$, (and as in Lemma~\ref{L:k0}, $C_F(A,0)$
has dimension $k_0n$ where $A$ has nullity $k_0$).
\end{remark}

It is natural to ask whether $C_F(A_\inv,\lambda)$ can be further decomposed.
We identify the \emph{general linear} group $\GL(F^n)$ with the
group $\GL(n,F)$ of $n\times n$ invertible matrices over~$F$.
The following is an isomorphism of codes, in particular it
preserves Hamming distances. 

\begin{theorem}\label{T:decom2}
Suppose $A\in F^{n\times n}$, $0\ne\lambda\in F$, and $\textup{char}(F)=p>0$.
Suppose $c_A(t)$ factors over $F$ as $(t-\alpha)^mf(t)$ where
$f(\alpha)f(\lambda\alpha)f(\lambda^{-1}\alpha)\ne0$. Then $A$ is conjugate to
$\textup{diag}(A_1,A_2)$ where $c_{A_1}(t)=(t-\alpha)^m$, $c_{A_2}(t)=f(t)$, and
\begin{equation}\label{E:decom2}
  C_F(A,\lambda)\cong C_F(A_1,\lambda)\oplus C_F(A_2,\lambda).
\end{equation}
\end{theorem}

\begin{proof}
Write $V=V_1\oplus V_2$ where $V_1=\{v\in V\mid v(A-\alpha I)^n=0\}$ and
$V_2=V(A-\alpha I)^n$. Then $A$ is conjugate to
$\textup{diag}(A_1,A_2)$ where $A_1\in\GL(V_1)$, $A_2\in\GL(V_2)$ and
$c_{A_2}(t)=f(t)$ is coprime to $c_{A_1}(t)=(t-\alpha)^m$.
By Theorem~\ref{T:decom} we may assume that $\alpha\ne0$,
and hence~$A_1$ is invertible. Since $A$ is conjugate to
$\textup{diag}(A_1\oplus(A_2)_\inv,(A_2)_\nil)$ where
$A_\inv=A_1\oplus (A_2)_\inv$ and $A_\nil=(A_2)_\nil$,
it suffices by Theorem~\ref{T:decom}
to prove~\eqref{E:decom2} when $A$ is invertible. Assume this is so.

Suppose $B\in C_F(A,\lambda)$, and write 
$B=\left(\begin{smallmatrix}B_1&B_2\\B_3&B_4\end{smallmatrix}\right)$
as in the proof of Theorem~\ref{T:decom}.
It follows from~\eqref{E:} (and the discussion following it) that 
$B_1\in C_F(A_1,\lambda)$, $B_4\in C_F(A_2,\lambda)$, and 
$B_3=\lambda^k A_2^{-k}B_3A_1^k$ for all $k\ge0$.
Take $k=p^m$ where $m=\lceil\log_p(n)\rceil$. Since $k\ge n$ and
$(A_1-\alpha I)^n=0$ we have $0=(A_1-\alpha I)^k=A_1^k-\alpha^k I$.
Therefore $B_3=(\lambda\alpha)^k A_2^{-k}B_3$. Suppose that this equation
has a solution $B_3\ne0$. Then $(\lambda\alpha)^k A_2^{-k}-I$ is singular,
and so too is $(\lambda\alpha)^kI- A_2^k$. This implies that
$(\lambda\alpha)^k$ is an eigenvalue of~$A_2^k$. Therefore 
$(\lambda\alpha)^k\in S(A_2^k)=S(A_2)^k$ where $S(A_2)^k$ denotes the
$k$th powers of the eigenvalues of~$A_2$. Taking $k$th roots
(which are unique in characteristic~$p$ as $k=p^m$) shows that
$\lambda\alpha\in S(A_2)$ and so
$0=c_{A_2}(\lambda\alpha)=f(\lambda\alpha)$, a contradiction. Thus the only solution is $B_3=0$.
Similar reasoning shows that $B_2=\lambda^{-k} A_1^{-k}B_2A_2^k$.
If $B_2\ne0$, then $(\lambda\alpha^{-1})^kA_2^k-I$ is singular and so too
is $(\lambda^{-1}\alpha)^kI-A_2^k$. This implies that $f(\lambda^{-1}\alpha)=0$,
contrary to our assumption. Hence $B_2=0$, and $B\mapsto (B_1,B_4)$
defines an $F$-linear isomorphism
$C_F(A,\lambda)\cong C_F(A_1,\lambda)\oplus C_F(A_2,\lambda)$.
\end{proof}

The following lemma is proved in~\cite[Proposition~2.1]{TCC} and used
for Theorem~\ref{T:NilpDim}.

\begin{lemma}\label{L:ExistInvertible}
If $A\in F^{n\times n}$, $\lambda\in F$ and $C_F(A,\lambda)$ contains an invertible matrix,~then
\[
  \dim C_F(A,\lambda)=\dim C_F(A).
\]
\end{lemma}

Let $J_\mu$ denote the $\mu\times\mu$ matrix with $(i,i+1)$ entry 1
for $1\le i<\mu$,
and zeros elsewhere. A nilpotent matrix $A\in F^{n\times n}$ has
\emph{Jordan form}
$\bigoplus_{i=1}^k J_{\mu_i}$ for some $k\ge1$ and some partition
$\mu=(\mu_1,\dots,\mu_k)$ of~$n$ with $\mu_1\ge\cdots\ge\mu_k\ge1$ and
$\sum_{i=1}^k \mu_i=n$. It is convenient to represent partitions by their
Young diagrams as in Figure~\ref{F:Young}.
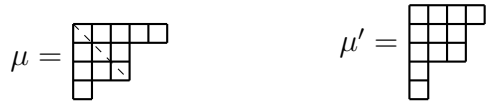
\begin{figure}[!ht]
  \caption{Young diagrams for $\mu=(5,3,3,1)$ and
    $\mu'=(4,3,3,1,1)$.}\label{F:Young}
  \begin{center}
  \begin{tikzpicture}[scale=0.25]
    \coordinate [label=left:\textcolor{black}{$\mu=$}] (A) at (0,2);
    \draw[thick,black](0,0)--(0,4);
    \draw[thick,black](1,0)--(1,4);\draw[thick,black](2,1)--(2,4);
    \draw[thick,black](3,1)--(3,4);
    \draw[thick,black](4,3)--(4,4);\draw[thick,black](5,3)--(5,4);
    \draw[thick,black](0,0)--(1,0);\draw[thick,black](0,1)--(3,1);
    \draw[thick,black](0,2)--(3,2);\draw[thick,black](0,3)--(5,3);
    \draw[thick,black](0,4)--(5,4);\draw[dashed] (0,4)--(3,1);
  \end{tikzpicture}
  \hskip20mm
  \begin{tikzpicture}[scale=0.25]
    \coordinate [label=left:\textcolor{black}{$\mu'=$}] (A) at (0,3);
    \draw[thick,black](0,0)--(0,5);\draw[thick,black](1,0)--(1,5);
    \draw[thick,black](2,2)--(2,5);\draw[thick,black](3,2)--(3,5);
    \draw[thick,black](4,4)--(4,5);
    \draw[thick,black](0,0)--(1,0);\draw[thick,black](0,1)--(1,1);
    \draw[thick,black](0,2)--(3,2);\draw[thick,black](0,3)--(3,3);
    \draw[thick,black](0,4)--(4,4);\draw[thick,black](0,5)--(4,5);
  \end{tikzpicture}
  \end{center}
\end{figure}
Reflecting the Young diagram of $\mu$ about its
main diagonal gives the Young diagram of the \emph{conjugate} partition
denoted $\mu'=(\mu'_1,\dots,\mu'_{k'})$, see Figure~\ref{F:Young}. Note that
$\mu'_i=|\{j\mid \mu_j\ge i\}|$ and  $k'=\mu_1$. If $A$ is conjugate to
$\bigoplus_{i=1}^k J_{\mu_i}$, then it follows
from~\cite[Theorem~6.1.3]{CM} that $\dim C_F(A) =\sum_{i=1}^{k'} (\mu'_i)^2$.
Further, $k=\mu'_1=k_0$ is the dimension of the null space of~$A$.

\begin{theorem}\label{T:NilpDim}
Suppose $A\in F^{n\times n}$ is nilpotent with nullity $k_0$ and
$\bigoplus_{i=1}^{k_0} J_{\mu_i}$ is its Jordan form where
$\sum_{i=1}^{k_0} \mu_i=n$. If $0\ne\lambda\in F$, then
$\dim C_F(A,\lambda)=\sum_{i=1}^{\mu_1} (\mu'_i)^2\ge k_0^2$.
\end{theorem}

\begin{proof}
As mentioned above, $A$ is conjugate in $\GL(n,F)$ to
$J=\bigoplus_{i=1}^{k_0} J_{\mu_i}$ and $k_0=\mu'_1$.
We now argue that
$C_F(A,\lambda)$ contains an invertible matrix. To see this note that the
invertible matrix
$D_{\mu_i}:=\textup{diag}(1,\lambda,\dots,\lambda^{\mu_i-1})$ lies in
$C_F(J_{\mu_i},\lambda)$. Thus $\bigoplus_{i=1}^{k_0} D_{\mu_i}$ is an
invertible element of $C_F(J,\lambda)$.
By Lemma~\ref{L:ExistInvertible} we have
$\dim C_F(A,\lambda)=\dim C_F(A)$, and as remarked before this theorem,
$\dim C_F(A) =\sum_{i=1}^{\mu_1} (\mu'_i)^2$ by~\cite[Theorem~6.1.3]{CM}.
\end{proof}

\begin{corollary}\label{C:det}
If $A\in F^{n\times n}$ is singular and $\lambda\in F$, then
$C_F(A,\lambda)\ne\{0\}$.
\end{corollary}

\begin{proof}
Theorems~\ref{T:decom} and~\ref{T:NilpDim} deal with $\lambda\ne0$, and
Lemma~\ref{L:k0} deals with $\lambda=0$.
\end{proof}

Let $A\in F^{n\times n}$ have rank $\rk(A)$. Since the nullity of $A$ equals
$n-\rk(A)$, Lemma~\ref{L:k0} implies that $\dim C_F(A,0)=n(n-\rk(A))$.
When $\lambda\ne0$, we will prove
that $\dim C_F(A,\lambda)$ equals $(n-\rk(A))^2+\delta$ where
$0\le\delta\le\rk(A)^2$. 

\begin{theorem}\label{T:null2rk2}
If $0\ne\lambda\in F$ and $A\in F^{n\times n}$ has rank $r$, then
\[
  (n-r)^2\le\dim C_F(A,\lambda)\le (n-r)^2+r^2.
\]
\end{theorem}

\begin{proof}
The nullity of $A$ is $k_0=n-r$. If $A_\nil$ has Jordan form
$\bigoplus_{i=1}^{\mu'_1}J_{\mu_i}$, then $k_0=\mu'_1$ and
$\dim C_F(A_\nil,\lambda)=\sum_{i=1}^{\mu_1}(\mu'_i)^2$ by Theorem~\ref{T:NilpDim}.
The lower bound follows from Theorem~\ref{T:decom} since:
\[
  \dim C_F(A,\lambda)\ge\dim C_F(A_\nil,\lambda)=\sum_{i=1}^{\mu_1}(\mu'_i)^2
    \ge(\mu'_1)^2=(n-r)^2.
\]
Also $\sum_{i=1}^{\mu_1}\mu'_i=m_0$, since $A_\nil$ has order $m_0$.
Hence $m_0-k_0\le\sum_{i=2}^{\mu_1}(\mu'_i)^2\le (m_0-k_0)^2$. Now $A_\inv$ has
order $n-m_0$ so $\dim C_F(A_\inv,\lambda)\le (n-m_0)^2\le r^2$. Therefore
\[
  \dim C_F(A,\lambda)\le k_0^2+(m_0-k_0)^2+(n-m_0)^2
\]
by Theorem~\ref{T:decom}. The upper bound follows from
$(m_0-k_0)^2+(n-m_0)^2\le (n-k_0)^2=r^2$.
\end{proof}

\begin{remark}
If $\rk(A)$ is small, then the bounds in Theorem~\ref{T:null2rk2} are
close. However, when the $\rk(A)$ is large, the upper bound of $n^2/2$ in
Theorem~\ref{T:UB} is better. Note that $n^2/2<(n-r)^2+r^2$ unless $r=n/2$.
\end{remark}

\begin{remark}
When $\rk(A)=1$, we can compute $\dim C_F(A,\lambda)$ precisely.
Suppose $\dim C_F(A,\lambda)=(n-1)^2+\delta$. If $\lambda=0$, then
$\delta=n-1$ by Lemma~\ref{L:k0}. Suppose $\lambda\ne0$ and use
Theorem~\ref{T:decom}. If $A$ is nilpotent, then $\mu=(2,1,\dots,1)$ so
$\mu'=(n-1,1)$ and thus $\delta=1$ by Theorem~\ref{T:NilpDim}. Otherwise
$A\ne A_\nil$, so $A_\nil=0$ has order $n-1$ and $A_\inv$ has order~1.
Thus $\delta=\dim C_F(A_\inv,\lambda)$ equals 1 if $\lambda=1$, and
$\delta=0$ if $\lambda\not\in\{0,1\}$.
\end{remark}

\begin{corollary}\label{C:rank1}
Suppose $A\in F^{n\times n}$ has rank~$1$ and $\lambda\in F$. Then
$\dim C_F(A,\lambda)$ equals $n(n-1)$ if $\lambda=0$, equals
$(n-1)^2$ if $\lambda\not\in\{0,1\}$ and $\textup{Tr}(A)\ne0$, and
$(n-1)^2+1$ otherwise.
\end{corollary}

\begin{proof}
Since $A$ has nullity $n-1$, the dimension is $n(n-1)$ when $\lambda=0$
by Lemma~\ref{L:k0}. Suppose now that
$\lambda\ne0$. We have $c_A(t)=t^{n-1}(t-\alpha)$
where $\alpha=\textup{Trace}(A)$. If $\alpha\ne0$, then $A$ preserves the 
primary decomposition $V=V_0\oplus V_\alpha$ and $A$ is conjugate to
$A_\nil\oplus A_\inv$ where $A_\nil=0$ has order $n-1$ and
$A_\inv=(\alpha)$ has order $1$.
Thus $\dim C_F(A,\lambda)=(n-1)^2+\delta$ by Theorems~\ref{T:decom} where
$\delta=1$ if $\lambda=1$, and $\delta=0$ otherwise.
Finally, suppose that $\lambda\ne0$ and $\textup{Trace}(A)=0$.
In this case $A$ is nilpotent and its Jordan form corresponds to the partition
$\mu=(2,1,\dots,1)$ of~$n$. Since the conjugate partition is $\mu'=(n-1,1)$,
Theorem~\ref{T:NilpDim} gives $\dim C_F(A,\lambda)=(n-1)^2+1$.
\end{proof}

\section{Parity check matrices for \texorpdfstring{$C_F(A,\lambda)$}{}}
\label{S:Parity}

\begin{lemma}\label{L:FL}
Suppose $L$ is an extension field of $F$.
Then $C_L(A,\lambda)\cong C_F(A,\lambda)\otimes_F L$.
\end{lemma}

\begin{proof}
One may use linear algebra to compute $C_F(A,\lambda)$. View $B=(b_{ij})$ as having
$n^2$ indeterminate entries. Then $AB-\lambda BA=0$ gives a consistent system
of $n^2$ homogeneous linear equations over $F$. The system has
$d:=\dim C_F(A,\lambda)$ free variables. Letting the free
variables range (independently) over $F$ gives all elements of $C_F(A,\lambda)$,
and similarly letting the free variables range over $L$ gives all elements of $C_L(A,\lambda)$. If $B\in C_F(A,\lambda)$ and $\mu\in L$, then
$\mu B\in C_L(A,\lambda)$, and the map
$C_F(A,\lambda)\otimes_F L\to C_L(A,\lambda)$ with
$B\otimes \mu\mapsto \mu B$ is a (well-defined) isomorphism.
Thus $C_L(A,\lambda)\cong C_F(A,\lambda)\otimes_F L$, as claimed.
\end{proof}

Given $B=(b_{ij})$ in $F^{n\times n}$, let $\Vec{B}$ be the $1\times n^2$ row
vector obtained by concatenating the rows of $B$. Its $k$th entry is
$(\Vec{B})_k=b_{ij}$ where $i=\lceil k/n\rceil$ and $j=k-(i-1)n$. The
$F$-linear map $\Vec{\phantom{n}}\colon F^{n\times n}\to F^{n^2}$
with $B\mapsto\Vec{B}$ is a vector space isomorphism.

\begin{lemma}\label{L:RNS}
With the above notation $\{\Vec{B}\mid B\in C_F(A,\lambda)\}$ equals
the row null space of $H=A^t\otimes I_n- \lambda I_n\otimes A$.
\end{lemma}

\begin{proof}
It suffices to prove that $\Vec{B}H=\Vec{AB-\lambda BA}$. A straightforward
calculation shows $\Vec{B}(\lambda I_n\otimes A)=\lambda\Vec{BA}$.
We now show that
$\Vec{B}(A^t\otimes I_n)=\Vec{AB}$. Suppose that $k=(i-1)n+j$ and
$k'=(i'-1)n+j'$.
Write $(I_n)_{ij}$ as $\delta_{ij}$. Then
\[
  (\Vec{B})_k=b_{ij},\quad (A\otimes I_n)_{kk'}=a_{ii'}\delta_{jj'},\quad\textup{and}\quad (A^t\otimes I_n)_{kk'}=a_{i'i}\delta_{jj'}.
\]
We use Einstein's convention that repeated subscripts 
are implicitly summed over to get
\[
  (\Vec{B}(A^t\otimes I_n))_{k'}=(\Vec{B})_{k}(A^t\otimes I_n)_{kk'}=b_{ij}a_{i'i}\delta_{jj'}
  =b_{ij'}a_{i'i}=a_{i'i}b_{ij'}=(AB)_{i'j'}.
\]
Therefore $\Vec{B}(A^t\otimes I_n)=\Vec{AB}$ and
$\Vec{B}(A^t\otimes I_n- \lambda I_n\otimes A)
  =\Vec{AB}-\lambda\Vec{BA}=\Vec{AB-\lambda BA}$.
\end{proof}

\begin{remark}
Lemma~\ref{L:RNS} can be paraphrased
$\Vec{C_F(A,\lambda)}=\textup{RNull}(A^t\otimes I_n- \lambda I_n\otimes A)$.
There is a corresponding result for column vectors.
For $B\in F^{n\times n}$ define $[B]\in F^{n^2\times 1}$ by $[B]^t=\Vec{B^t}$.
Transposing 
$\Vec{B}(A^t\otimes I_n- \lambda I_n\otimes A)=\Vec{0}$ gives
$(I_n\otimes A- \lambda A^t\otimes I_n)[B^t]=[0]$.
Therefore
$[C_F(A,\lambda)^t]=\textup{CNull}(I_n\otimes A- \lambda A^t\otimes I_n)$.
The matrix $H$ in Lemma~\ref{L:RNS} (and its transpose) are called
\emph{parity check matrices}. Although na\"{\i}vely checking whether 
$\Vec{B}H=\Vec{0}$ is computationally expensive (as $H$ is $n^2\times n^2$),
this equation is equivalent to
$AB-\lambda BA=0$ which can be computed at a cost (essentially)
of two multiplications of $n\times n$ matrices. 
\end{remark}

The eigenvalues of $H$ are related to $\lambda$ and the eigenvalues of $A$
as follows.

\begin{proposition}\label{P}
Suppose $A\in F^{n\times n}$, $\lambda\in F$ and $c_A(t)=\prod_{j=1}^n(t-\alpha_j)$
where the $\alpha_j\in\overline{F}$. Then
$H=A^t\otimes I_n -\lambda I_n\otimes A $ has characteristic polynomial
\[ 
  c_H(t)=\prod_{i=1}^n\prod_{j=1}^n(t-(\alpha_i-\lambda\alpha_j)).
\]
\end{proposition}

\begin{proof}
Suppose $R\in F^{m\times m}$, $S\in F^{n\times n}$ and
$c_R(t)$ and $c_S(t)$ factor as $\prod_{i=1}^m(t-\rho_i)$
and $\prod_{j=1}^n(t-\sigma_j)$ respectively in $\overline{F}$.
Now $H_1=R\otimes I_n+I_m\otimes S$ is conjugate in $\GL(mn,\overline{F})$ to
$H_2=J(R)\otimes I_n+I_m\otimes J(S)$ where $J(R)$ and $J(S)$ are the
Jordan forms of $R$ and $S$.
Thus $c_{H_1}(t)=c_{H_2}(t)=\prod_{i=1}^m\prod_{j=1}^n(t-(\rho_i+\sigma_j))$.
Now set $R=A^t$ and $S=-\lambda A$.
\end{proof}

\section{The cases when \texorpdfstring{$A$}{} is cyclic}\label{S:cyclic}

Given polynomials $f(t)=\prod_{i=1}^m(t-\alpha_i)$ and
$g(t)=\prod_{j=1}^n(t-\beta_j)$ define $(f\otimes g)(t)$ to be
$\prod_{i=1}^m\prod_{j=1}^n(t-\alpha_i\beta_j)$. It is easy to prove that
$(f\otimes g)\otimes h=f\otimes (g\otimes h)$, $f\otimes (t-1)=f$,
$(fg)\otimes h=(f\otimes h)(g\otimes h)$, see~\cite{SG}.

For $0\ne\lambda\in F$ define the
\emph{$\lambda$-twisted characteristic polynomial} of $A\in F^{n\times n}$ to be
\[
  c_{A,\lambda}(t):=c_A(t)\otimes(t-\lambda^{-1})=\lambda^{-n}c_A(\lambda t).
\]
If
$c_A(t)=\prod_{\alpha\in S(A)}(t-\alpha)^{m_\alpha}=t^n+\sum_{i=0}^{n-1}\gamma_it^i$,
then
\[
  c_{A,\lambda}(t)=t^n+\sum_{i=0}^{n-1}\gamma_i\lambda^{-(n-i)}t^i
  =\prod_{\alpha\in S(A)}\left(t-\lambda^{-1}\alpha\right)^{m_\alpha}.
\]
Since
$\gcd(f,g)\otimes (t-\lambda)=\gcd(f\otimes (t-\lambda),g\otimes (t-\lambda))$,
and $(t-\lambda^{-1})\otimes(t-\lambda)=t-1$, we have
\begin{align*}
  \gcd(c_A(t),c_{A,\lambda}(t))\otimes(t-\lambda)
  &=\gcd(c_A(t)\otimes(t-\lambda),c_A(t)\otimes(t-\lambda^{-1})\otimes(t-\lambda))\notag\\
  &=\gcd(c_{A,\lambda^{-1}}(t),c_A(t))\label{E:ten}.
\end{align*}
Therefore 
$\deg(\gcd(c_A(t),c_{A,\lambda}(t)))=\deg\gcd(c_{A,\lambda^{-1}}(t),c_A(t))$, as stated in Theorem~\ref{T:Zero}, in
the second equality of~\eqref{E:deg}. The first equality proof of~\eqref{E:deg}
uses the following lemma.

A matrix $A\in F^{n\times n}$ is called \emph{cyclic} if it has a cyclic vector $v$, that is $v, vA,\dots, vA^{n-1}$
is a basis for~$F^n$. 
Equivalently, $c_A(t)$ equals the minimal polynomial of $A$ by~\cite[\S12]{HH}.

\begin{lemma}\label{L:HH}
Given a cyclic matrix $A\in F^{n\times n}$ and a polynomial
$g(t)\in F[t]$, the linear map $\phi_g\colon F^n\to F^n$ defined by
$v^{\phi_g}=vg(A)$ has kernel of dimension $\deg(\gcd(c_A(t),g(t)))$. 
\end{lemma}

\begin{proof}
As $A$ is cyclic, there exists a
vector $e_0\in V:=F^n$ such that $e_0,e_0A,\dots,e_0A^{n-1}$ is a basis for~$V$.
Thus each $v\in V$ equals $e_0h(A)$ for a unique
$h(t)=\sum_{i=0}^{n-1}h_it^i\in F[t]$ of degree at most $n-1$.
Observe that $e_0h(A)$ lies in $\ker(\phi_g)$ precisely when
\[
  e_0h(A)g(A)=0 \Leftrightarrow e_0A^ih(A)g(A)=0\textup{ for $0\le i\le n-1$}
  \Leftrightarrow h(A)g(A)=0.
\]
In summary, $e_0h(A)\in\ker(\phi_g)\Leftrightarrow c_A(t)\mid h(t)g(t)
\Leftrightarrow c_A(t)/d(t)$ divides $h(t)$ where
$d(t)$ is defined to be $\gcd(c_A(t),g(t))$. It follows that the space
of possible $h(t)$ has dimension $\deg(d(t))$ and hence
$\dim \ker(\phi_g)=\deg(d(t))$. 
\end{proof}

The next result explores twisted centraliser codes for cyclic matrices.

\begin{theorem}\label{T:Zero}
If $A\in F^{n\times n}$ is cyclic, and $\lambda\ne0$, then 
\begin{equation}\label{E:deg}
  \dim C_F(A,\lambda)=\deg(\gcd(c_A(t),c_{A,\lambda}(t)))
  =\deg(\gcd(c_{A,\lambda^{-1}}(t),c_A(t))).
\end{equation}
\end{theorem}

\begin{proof}
Since $A$ is cyclic, there is a basis $e_0,e_1,\dots,e_{n-1}$ for the
row space $V=F^n$, where $e_i=e_0A^i$ for $0\le i<n$. Without loss 
of generality write $A$ with respect to this basis. Let $b$ be the first
row of $B\in C_F(A,\lambda)$, so $b=e_0B$.
Since $AB-\lambda BA=0$ we see that $e_iAB=\lambda e_iBA$, that is,
$e_{i+1}B=\lambda e_iBA$. A simple
induction shows that $e_iB=\lambda^ibA^i$ for $0\le i\le n-1$. 
Thus $B\in C_F(A,\lambda)$ is determined by the vector $b$.
On the other hand let $b\in F^n$ and let $B\in F^{n\times n}$
be the matrix with rows $b, \lambda bA, \dots, \lambda^{n-1}bA^{n-1}$.
Then, for $0\le i\le n-2$, we clearly have
\[ 
  e_i(AB-\lambda BA)=e_{i+1}B-\lambda e_iBA
   =\lambda^{i+1}bA^{i+1}-\lambda(\lambda^ibA^i)A=0.
\]
It remains to determine when $e_{n-1}(AB-\lambda BA)=0$.
Since $e_0,e_0A,\dots,e_0A^{n-1}$ is a basis for $V$,
there exist scalars $\gamma_0,\dots,\gamma_{n-1}\in F$ such that
$e_0A^n=-\sum_{i=0}^{n-1}\gamma_ie_0A^i$. Indeed, the characteristic
polynomial of $A$ is $c_A(t)=t^n+\sum_{i=0}^{n-1}\gamma_it^i$. Furthermore,
\begin{align*}
  e_{n-1}AB&=e_0A^nB=-\sum_{i=0}^{n-1}\gamma_ie_0A^iB=-\sum_{i=0}^{n-1}\gamma_ie_iB
  =-\sum_{i=0}^{n-1}\gamma_i\lambda^ibA^i,\quad\textup{and}\\
  \lambda e_{n-1}BA&=\lambda(\lambda^{n-1}bA^{n-1})A=\lambda^nbA^n.
\end{align*}
Thus the equation $e_{n-1}(AB-\lambda BA)=0$ is equivalent to
\[
  0=-\lambda^nbA^n-\sum_{i=0}^{n-1}\gamma_i\lambda^ibA^i
  =-bc_A(\lambda A)=-\lambda^nbc_{A,\lambda}(A).
\]
Hence we conclude that $\dim C_F(A,\lambda)$ equals the dimension of the row
null space of the matrix $c_{A,\lambda}(A)$. Since $A$ is cyclic, it follows
from Lemma~\ref{L:HH} that this null space has
dimension $\deg(\gcd(c_A(t),c_{A,\lambda}(t)))$. 
\end{proof}

\begin{remark}
If $e_0, e_0A,\dots,e_0A^{n-1}$ is a basis for $F^n$, the proof of
Theorem~\ref{T:Zero} gives a vector space isomorphism
$ C_F(A,\lambda)\to\textup{RNull}(c_{A,\lambda}(A))$ given by $B\mapsto e_0B$.
\end{remark}

\begin{remark}
Every matrix $A$ is conjugate to a sum $A_1\oplus\cdots\oplus A_r$
of cyclic matrices by~\cite{HH}. Thus Theorem~\ref{T:Zero} gives 
$\dim C_F(A,\lambda)\ge\sum_{i=1}^r \deg(\gcd(c_{A_i}(t),c_{A_i,\lambda}(t)))$.
\end{remark}

\begin{lemma}\label{L:min}
$\dim C_F(A)\ge\deg(m_A(t))$ with equality if and only if $A$ is cyclic.
\end{lemma}

\begin{proof}
It is clear that $F[A]\le C_F(A)$ where $F[A]=\{f(A)\mid f(t)\in F[t]\}$.
Hence $\deg( m_A(t))=\dim(F[A])\le \dim C_F(A)$.  Equality holds if and only if
$F[A]= C_F(A)$, and this is true if and only if $A$ is cyclic
by~\cite[Theorem~2.1]{NeuP}.
\end{proof}

Recall from Section~\ref{S:Intro} that $L$ is the splitting field for $c_A(t)$,
and $S(A)$ is the
set of roots of $c_A(t)$, also called the \emph{spectrum} of $A$.
Recall that $c_A(t)=\prod_{\alpha\in S(A)}(t-\alpha)^{m_{A,\alpha}}$ and
$L^n=\bigoplus_{\alpha\in S(A)} M_{A,\alpha}$. Also $K_{A,\alpha}$ is the
$\alpha$-eigenspace of $A$, and $K_{A,\alpha}\le M_{A,\alpha}$.

\begin{proposition}\label{PK}
\textup{\cite[Theorem 66]{K}}
Let $F$ be an arbitrary field, and let $A$ be an $n\times n$ matrix
over~$F$. Then $A$ is conjugate to $A^t$ via a symmetric matrix.
\end{proposition}

Proposition~\ref{PK} implies that $k_{A^t,\alpha}=k_{A,\alpha}$ and 
$m_{A^t,\alpha}=m_{A,\alpha}$ for all $\alpha\in\overline{F}$.

It follows from Theorem~\ref{T:NilpDim} that the possibility
$\dim C_F(A,\lambda)>k_0^2$ occurs when $\mu_1>1$,
$A\in F^{n\times n}$ is nilpotent,
and $0\ne\lambda\in F$. This proves that Theorem~2.4 of~\cite{CC}
(which considers the case $\lambda=1$) is wrong. A correct
generalization is below.

\begin{theorem}\label{Tm}
Suppose $\lambda\in F$ and $A\in F^{n\times n}$. Then 
\[
  \sum_{\alpha\in S(A)}k_{A,\lambda\alpha}k_{A,\alpha}
  \le\dim_F C_F(A,\lambda)\le\sum_{\alpha\in S(A)}m_{A,\lambda\alpha}m_{A,\alpha}.
\]
\end{theorem}

\begin{proof}
We use the isomorphism $\Vec{\phantom{n}}$ introduced before Lemma~\ref{L:RNS}.
Let $L$ be a splitting field for $c_A(t)$, and hence for $c_{A,\lambda}(t)$.
Since
$\dim_F C_F(A,\lambda)=\dim_L C_L(A,\lambda)$ by Lemma~\ref{L:FL},
we may assume that $F=L$. Suppose $uA^t=\alpha u$ and $vA=\beta v$ where
$\alpha,\beta\in S(A)=S(A^t)$.
Then $(u\otimes v)(A^t\otimes I-\lambda I\otimes A)
=(\alpha-\lambda\beta)(u\otimes v)$.
Take $\alpha=\lambda\beta$. Then $(u\otimes v)H=0$ and hence
$K_{A^t,\lambda\beta}\otimes K_{A,\beta}\le \Vec{C_L(A,\lambda)}$ by Lemma~\ref{L:RNS}. 
Observe that
$L^n=\bigoplus_{\alpha\in S(A^t)}M_{A^t,\alpha}=\bigoplus_{\beta\in S(A)}M_{A,\beta}$
and hence
$L^n\otimes L^n=\bigoplus_{\alpha,\beta}M_{A^t,\alpha}\otimes M_{A,\beta}$.
Thus the sum
$\sum_{\beta\in S(A^t)}K_{A^t,\lambda\beta}\otimes K_{A,\beta}\le \Vec{C_L(A,\lambda)}$
is direct. Taking dimensions and using $k_{A^t,\alpha}=k_{A,\alpha}$ establishes
the lower bound.

Suppose that $B\in C_L(A,\lambda)$. Then
$\Vec{B}=\sum_{\alpha,\beta} u_\alpha\otimes v_{\beta}$ for some
$u_\alpha\in M_{A^t,\alpha}$ and
$v_{\beta}\in M_{A,\beta}$. Since
$0=\Vec{B}H=\sum_{\alpha,\beta} (\alpha-\lambda\beta)u_\alpha\otimes v_{\beta}$,
it follows that $(\alpha-\lambda\beta)u_\alpha\otimes v_{\beta}=0$ for each
$\alpha,\beta$. If for some $\alpha,\beta$ we have
$\alpha\ne\lambda\beta$, then $u_\alpha\otimes v_{\beta}=0$.
Thus $\Vec{B}=\sum_{\beta} u_{\lambda\beta}\otimes v_{\beta}$ and it follows that
$\Vec{C_L(A,\lambda)}\le\bigoplus_{\beta\in S(A^t)}M_{A^t,\lambda\beta}\otimes M_{A,\beta}$.
Taking dimensions and using $m_{A^t,\alpha}=m_{A,\alpha}$ establishes
the upper bound.
\end{proof}

A matrix $A\in F^{n\times n}$ is called \emph{semisimple} if every $A$-invariant
subspace of $F^n$ has an $A$-invariant complement, or equivalently, the minimal polynomial of $A$ is separable. Recall that a polynomial over $F$ is
\emph{separable} if it has distinct roots in the algebraic closure $\overline{F}$; and separable means squarefree when $F$ is perfect.

\begin{corollary}\label{C:TFAE}
Suppose $\lambda\in F$, $A\in F^{n\times n}$ and $L$ is a splitting
field of $c_A(t)$. The following are equivalent: 
\begin{enumerate}[{\rm (a)}]
  \item $A$ is conjugate in $\GL(n,F)$ to the block
    matrix $A_\inv\oplus 0$ where $A_\inv$ is semisimple;
  \item $A$ is conjugate in $\GL(n,L)$ to a diagonal matrix;
  \item $\dim_F C_F(A,\lambda)=\sum_{\alpha\in S(A)}k_{A,\lambda\alpha}k_{A,\alpha}
    =\sum_{\alpha\in S(A)}m_{A,\lambda\alpha}m_{A,\alpha}$.
\end{enumerate}
\end{corollary}

\begin{proof}
If (a) is true, then so is (b) by the definition of semisimple. If (b) is
true, then $k_{A,\alpha}=m_{A,\alpha}$ for each root $\alpha$ of $c_A(t)$. Hence
Theorem~\ref{Tm} implies (c). If (c) is true, then $k_{A,\alpha}=m_{A,\alpha}$
for each $\alpha\in S(A)$ and so $A$ is diagonalizable over $L$, and hence
$A_\inv$ and $A_\nil$ are semisimple over $F$. The latter implies $A_\nil=0$,
and so (a) is true.
\end{proof}

\section{The probability that \texorpdfstring{$C_F(A,\lambda)\ne\{0\}$}{}}
\label{S:Quokka}

In this section we estimate the probability that a uniformly
distributed $A\in\F_q^{n\times n}$ has $C_{\F_q}(A,\lambda)\ne\{0\}$.
We will see that the (abundant) invertible matrices
contribute a small amount to the probability, and the (rare) singular matrices
contribute a lot. 

\begin{lemma}\label{L:1/q}
The probability that a uniformly distributed $A\in \F_q^{n\times n}$
has $C_F(A,\lambda)\ne\{0\}$ is at least $q^{-1}$ (because
$C_F(A,\lambda)\ne\{0\}$ whenever $A$ is singular).
\end{lemma}

\begin{proof}
By  Corollary~\ref{C:det}, $C_F(A,\lambda)\ne\{0\}$ is true
for every singular matrix $A$. Thus the probability 
that  a uniformly distributed $A\in \F_q^{n\times n}$
has $C_F(A,\lambda)\ne\{0\}$ is at least the probability that
$A$ is singular. The latter probability is $1- |\GL(n,q)|/q^{n^2}$.

Moreover, by~\cite[Lemma~3.5]{NeuP},
\begin{equation}\label{E:GL}
  1-q^{-1}-q^{-2}<\frac{|\GL(n,q)|}{q^{n^2}}\le 1-q^{-1}
\end{equation}
and hence the probability that $A$ is singular
is at least $q^{-1}$.
\end{proof}

Recall the definition of a semisimple matrix given before Corollary~\ref{C:TFAE}. 
A matrix $U\in F^{n\times n}$ is called \emph{unipotent} if $(U-I)^n=0$, i.e.
$c_U(t)=(t-1)^n$. Each matrix $A\in \GL(n,q)$ has a unique `Jordan
 decomposition' $A=SU=US$ with $S$ semisimple and $U$ unipotent in
 $\GL(n,q)$ (see \cite[p.11]{RC} or \cite[Section 1.1]{NP}). We sometimes
refer to $S$ and $U$ as the semisimple and unipotent parts of $A$, respectively.

\begin{theorem}
Suppose $F=\F_q$ and $\lambda\in F$. Let $\pi$ be
the probability that a uniformly distributed $A\in F^{n\times n}$
has $C_F(A,\lambda)\ne\{0\}$. Then $\pi\ge q^{-1}$ if $\lambda=0$, and
$\pi=1$ if $\lambda=1$.
For $\lambda\not\in\{0,1\}$, if $n\ne 3$ then
$\pi\ge \frac54 q^{-1}$, while if $n=3$ then
$\pi\ge \frac76 q^{-1}$.
\end{theorem}

\begin{proof}
{\sc Case $\lambda=0$.} By Lemma~\ref{L:k0}, $C_F(A,\lambda)\ne\{0\}$
if and only if $A$ is singular. Thus $\pi\ge q^{-1}$ by Lemma~\ref{L:1/q}.

{\sc Case $\lambda=1$.} Since $C_F(A,1)$ always contains the scalar matrices,
we see $\pi=1$.

{\sc Case $\lambda\not\in\{0,1\}$.} In this case
$q\ge|\{0,1,\lambda\}|=3$. Let $C$ denote the subset of matrices 
$A\in\GL(n,F)$ such that $c_A(t)$ is divisible by 
$(t-\alpha)(t-\lambda \alpha)$, for some non-zero $\alpha\in F$. It follows 
from Theorem~\ref{T:Zero} that $C_F(A,\lambda)\ne \{0\}$ for each such matrix $A$.
We will prove that $\frac{|C|}{|\GL(n,F)|}\geq \frac{1}{\delta(q-1)}$, where $\delta = 4$ if $n\ne3$ and $\delta = 6$ if $n=3$.

 We note that $C$ is closed under conjugation by elements of
 $\GL(n,q)$.  Also if an element $A\in\GL(n,q)$ has `Jordan
 decomposition' $A=SU=US$ with $S$ semisimple and $U$ unipotent in
 $\GL(n,q)$, then $c_A(t)=c_S(t)$, and hence $A\in C$ if and only if
 $S\in C$. Thus $C$ is a quokka set, as defined in
 \cite[Definition 1.1]{NP}. Let $A\in C$ with $c_A(t)$ divisible by
 $(t-\alpha)(t-\lambda\alpha$), for some non-zero $\alpha\in F$. The
 semisimple part $S$ of $A$ preserves a direct sum decomposition of
 $V$ of the form $V_1\oplus V_2\oplus V_3$, where $V_1=\langle
 u\rangle$ with $uS = \alpha u$ , and $V_2=\langle v\rangle$ with $vS
 = \lambda \alpha v$. Thus we may assume that $S =
 \left(\begin{smallmatrix}S_1&0\\0&S_2\end{smallmatrix}\right)$ with
   $S_1 = \left(\begin{smallmatrix}\alpha&0\\0&\lambda
     \alpha\end{smallmatrix}\right)$ and $S_2$ semisimple in
     $\GL(n-2,q)$. 
The matrix $S$ is contained in a maximal abelian subgroup $T$ of $\GL(n,q)$
of the form $T\cong C_{q-1} \times C_{q-1} \times T'$, where $T'$ is a maximal abelian
subgroup of $\GL(n-2,q)$, and $T$ consists of matrices of the form
\[
\left(\begin{smallmatrix}
x&0&0\\ 
0&y&0\\
0&0&X\end{smallmatrix}\right)\qquad\textup{with $x,y\in F\setminus\{0\}$ and $X\in T'\le\GL(n-2,F)$}.
\]
The subgroup $T$ is sometimes called a \emph{maximal torus} of $\GL(n,q)$ as described in~\cite[Section~1.1]{NP}, and $T$ corresponds to a conjugacy class of permutations in $S_n$ with at least two fixed points because of the two direct factors $C_{q-1}$. Moreover, each such conjugacy class of $S_n$ corresponds to a maximal abelian subgroup $T$ of this form. Each of these matrices with $y=\lambda x$ or $y=\lambda^{-1}x$ belongs to $C$. Thus $|T\cap C|/|T|\geq \frac{1}{q-1}$ (and even $|T\cap C|/|T|\geq \frac{2}{q-1}$ if $\lambda\ne -1$).

It now follows from \cite[Theorem 1.3]{NP} that 
$|C|/|\GL(n,q)|\geq  \frac{p(n)}{q-1}$, where $p(n)$ is the proportion of permutations in $S_n$ with at least two fixed points. A direct computation shows that $p(n)\geq 1/4$ for $n=2, 4, 5$ while $p(3)=1/6$. Moreover, by \cite[p.159]{Aigner}, 
the proportion of derangements in $S_n$ is $\sum_{k=0}^{n} (-1)^k/k!$,  and the proportion of elements of $S_n$ with exactly one fixed point is 
$\sum_{k=0}^{n-1} (-1)^k/k!$, and hence $p(n) = 1- 2 \sum_{k=0}^{n-1} (-1)^k/k! - (-1)^n/n!$. Thus for $n\geq6$,
\[ 
p(n)\geq 1 - 2( 1 - 1 + \frac{1}{2}-\frac{1}{6}+\frac{1}{24}) = \frac{1}{4}. 
\]
We have therefore proved that $|C|/|\GL(n,q)|\geq  \frac{1}{\delta(q-1)}$ (or $\frac{2}{\delta(q-1)}$ if $\lambda\ne -1$) where $\delta=4$ 
if $n\ne 3$, and $\delta = 6$ when $n=3$.

The set $S$ of singular vectors is disjoint from $C$. Using~\eqref{E:GL}
and the above bound~gives
\begin{align*}
  \pi&\ge \frac{|S|+|C|}{q^{n^2}}
  =\frac{q^{n^2}-|\GL(n,q)|}{q^{n^2}}
    +\frac{|\GL(n,q)|}{q^{n^2}}\frac{|C|}{|\GL(n,q)|}\\
  &= 1-\frac{|\GL(n,q)|}{q^{n^2}}\left(1-\frac{|C|}{|\GL(n,q)|}\right)
  \ge 1-(1-q^{-1})\left(1-\frac{1}{\delta(q-1)}\right)\\
  &=1-(1-q^{-1})+\frac{q^{-1}}{\delta}=\frac{(\delta+1)q^{-1}}{\delta}
  =\begin{cases}
    \frac54 q^{-1}\quad&\textup{if $n\ne 3$,}\\
    \frac76 q^{-1}\quad&\textup{if $n=3$.}\\
    \end{cases}
\end{align*}
This completes the proof.
\end{proof}

\section{Upper bounds for \texorpdfstring{$\dim C_F(A,\lambda)$}{}}\label{S:UB}

For a positive integer $r$, and let~$[r]$ denote the set $\{1,2,\dots,r\}$.
A permutation $\sigma$ of~$[r]$ is called a \emph{derangement} if it
has no fixed points (i.e. $i\sigma\ne i$ for $i\in[r]$).

\begin{lemma}\label{L:der}
Suppose $1\le s\le r$ and $\sigma\colon [s]\to [r]$ is an injective map
with no fixed points. Then $\sigma$ extends to a derangement of 
a subset of $[r]$ of size $r-1$ or $r$.
\end{lemma}

\begin{proof}
Set $X_1=[s]$, $X_2=[r]\setminus X_1$, $Y_1=[s]\sigma$, and
$Y_2=[r]\setminus Y_1$.
Define the partition $[r]=Z_{11}\cup Z_{12}\cup Z_{21}\cup Z_{22}$
as in Figure~\ref{F:Venn}.
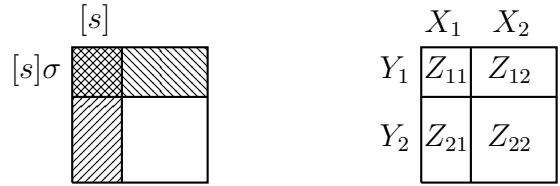
\begin{figure}[!ht]
  \caption{A partition $[r]=Z_{11}\cup Z_{12}\cup Z_{21}\cup Z_{22}$ with $Z_{ij}=Y_i\cap X_j$.}\label{F:Venn}
  \vskip3mm
  \begin{center}
  \begin{tikzpicture}[scale=0.6]
    \coordinate [label=above:\textcolor{black}{$[s]$}] (X1) at (0.5,3);
    \coordinate [label=left:\textcolor{black}{$[s]\sigma$}] (Y2) at (0,2.5);
    \pattern[pattern=north east lines] (0,0) rectangle (1.1,3);
    \pattern[pattern=north west lines] (0,1.9) rectangle (3,3);
    \draw[thick,black](0,0)--(0,3)--(3,3)--(3,0)--(0,0);
    \draw[thick,black](0,1.9)--(3,1.9);
    \draw[thick,black](1.1,0)--(1.1,3);
  \end{tikzpicture}
  \hskip20mm
  \begin{tikzpicture}[scale=0.6]
    \coordinate [label=left:\textcolor{black}{$Z_{21}$}] (S1) at (1.3,1);
    \coordinate [label=left:\textcolor{black}{$Z_{11}$}] (S2) at (1.3,2.5);
    \coordinate [label=left:\textcolor{black}{$Z_{12}$}] (T1) at (2.7,2.5);
    \coordinate [label=left:\textcolor{black}{$Z_{22}$}] (T2) at (2.7,1);
    \coordinate [label=above:\textcolor{black}{$X_1$}] (X1) at (0.5,3);
    \coordinate [label=above:\textcolor{black}{$X_2$}] (X2) at (2,3);
    \coordinate [label=left:\textcolor{black}{$Y_2$}] (Y1) at (0,1);
    \coordinate [label=left:\textcolor{black}{$Y_1$}] (Y2) at (0,2.5);
    \draw[thick,black](0,0)--(0,3)--(3,3)--(3,0)--(0,0);
    \draw[thick,black](0,1.9)--(3,1.9);
    \draw[thick,black](1.1,0)--(1.1,3);
  \end{tikzpicture}
  \end{center}
\end{figure}
Note that $Z_{12}=Z_{21}\sigma$ and $|Z_{12}|=|Z_{12}|$.
We try to extend $\sigma$, when possible, by a fixed-point-free bijection
$\tau\colon Z_{12}\cup Z_{22}\to Z_{21}\cup Z_{22}$,
to a derangement of $[r]$. We define $\tau$ with two maps, one with
$Z_{12}\to Z_{21}$ which always works, and one with $Z_{22}\to Z_{22}$ which
works if $|Z_{22}|\ne1$. The first part of $\tau$ is the
inverse of the map $Z_{12}\to Z_{21}\colon i\mapsto i\sigma$.
This is fixed-point-free since $Z_{12}\cap Z_{21}=\emptyset$.
Take the second part of $\tau$ to be a derangement
of $Z_{22}$; this exists if and only if $|Z_{22}|\ne1$.
If $|Z_{22}|=1$, then $\sigma$ extended by $\tau$ gives a
derangement of the subset $[r]\setminus Z_{2,2}$ of cardinality $r-1$.
\end{proof}

Recall that $m_0:=\dim \{v\in F^n\mid vA^n=0\}$ and
$k_0:=\dim \{v\in F^n\mid vA=0\}$. In the following
theorem a lower bound for the rank $\rk(A)=n-k_0$ gives an upper bound for
$\dim C_F(A,\lambda)$, \emph{c.f.} Corollary~\ref{C:rank1}.

\begin{theorem}\label{T:UB}
Suppose that $A\in F^{n\times n}$ and $\lambda\in F$ is not $0$ or $1$.
If $k_0+m_0/2\le n$, then $\dim C_F(A,\lambda)\le n^2/2$.
In particular, $\dim C_F(A,\lambda)\le n^2/2$ holds for invertible matrices $A$.
\end{theorem}

\begin{proof}
We will prove below that
$\dim C_F(A,\lambda)\le n^2/2$ holds for all $A\in\GL(n,F)$.
Suppose that this is true, and suppose that $A\in F^{n\times n}$ is singular.
Then $A_\nil$ has order $m_0\ge1$
and $A_\inv$ has order $n-m_0$.
By Theorem~\ref{T:decom} and Lemma~\ref{L:k0} we have
$C_F(A,\lambda)\cong C_F(A_\nil,\lambda)\oplus C_F(A_\inv,\lambda)$ and
$\dim C_F(A_\nil,\lambda)\le k_0m_0$. Also by our assumption,
$\dim C_F(A_\inv,\lambda)\le (n-m_0)^2/2$.
Thus $\dim C_F(A,\lambda)\le k_0m_0+(n-m_0)^2/2$. This is at most $n^2/2$
if and only if $k_0+m_0/2\le n$.

It remains to prove $\dim C_F(A,\lambda)\le n^2/2$ for all $A\in\GL(n,F)$.
In this case $m_0=k_0=0$, and the bound $k_0+m_0/2\le n$ holds automatically.
By Theorem~\ref{Tm}, we have
\[
  \dim C_F(A,\lambda)\le\sum_{\alpha\in S(A)}m_{A,\lambda\alpha}m_{A,\alpha}
  \qquad\textup{where $c_A(t)=\prod_{\alpha\in S(A)}(t-\alpha)^{m_{A,\alpha}}$.}
\]
If $m_{A,\lambda\alpha}m_{A,\alpha}\ne0$, then $\alpha$ and $\lambda\alpha$
lie in $S(A)$. 
Therefore
$\dim C_F(A,\lambda)\le\sum_{\alpha\in\Gamma}m_{A,\lambda\alpha}m_{A,\alpha}$
where $\Gamma:=\{\alpha\in S(A)\mid \lambda\alpha\in S(A)\}$.
If $\Gamma=\emptyset$, then $\dim C_F(A,\lambda)\le0\le n^2/2$.
Assume now that $\Gamma\ne\emptyset$.

We simplify the notation by writing
$c_A(t)=\prod_{i=1}^r(t-\alpha_i)^{m_i}$ where $S(A)=\{\alpha_1,\dots,\alpha_r\}$.
Without loss of generality, suppose $\Gamma=\{\alpha_1,\dots,\alpha_s\}$
where $1\le s\le r$. 
Consider the map $\sigma\colon \{1,\dots,s\}\to\{1,\dots,r\}$ defined
by $\alpha_{i\sigma}=\lambda\alpha_i$. Since $\lambda\ne0$,
$\sigma$ is injective.
If $i\sigma=i$, then $\lambda\alpha_i=\alpha_i$, and since $\lambda\ne1$,
it follows that $\alpha_i=0$. This is a contradiction, as $\det(A)\ne0$.
Thus $\sigma$ has no fixed points, and 
$\dim C_F(A,\lambda)\le\sum_{i=1}^sm_{i\sigma}m_i$ where $n=\sum_{i=1}^rm_i$
by Theorem~\ref{Tm}.

Note that $\Gamma$ is the set of roots of $\gcd(c_A(t),c_{A,\lambda^{-1}}(t))$
because $c_A(t)$ has roots $\alpha_i$ and $c_{A,\lambda^{-1}}(t)$ has roots
$\lambda\alpha_i$. Since $1\le s\le r$, and $\sigma$ has no fixed points,
we have $r\ge2$. If $r=2$, then $\sigma$ must be the transposition $(1,2)$.
Thus 
\[
  \dim C_F(A,\lambda)\le\sum_{i=1}^sm_{i\sigma}m_i=m_2m_1+m_1m_2
  =2m_1m_2=2m_1(n-m_1)
  \le\frac{n^2}{2}.
\]
Suppose now that $r\ge3$. By Lemma~\ref{L:der} there is a derangement
$\tau$ of a $k$-subset of $\{1,\dots,r\}$
with $s\le k$, $k\in\{r-1,r\}$ and $i\tau=i\sigma$
for $1\le i\le s$.
Suppose that $\tau$ has one cycle. Without loss of generality, suppose
that $\tau=(1,2,\dots,k)$. Then
$\sum_{i=1}^sm_{i\sigma}m_i\le\sum_{i=1}^km_{i\tau}m_i$.
Suppose further that $k=2$ and hence $r=3$.
Then 
\[
  \sum_{i=1}^km_{i\tau}m_i=m_2m_1+m_1m_2 \le\frac{(m_1+m_2)^2}{2}<\frac{n^2}{2}
\]
as desired.
Suppose now that $3\le k\le r$. If follows from
\[
  n^2=\left(\sum_{i=1}^r m_i\right)^2
  =\sum_{i=1}^r m_i^2+2\sum_{i>j} m_im_j\ge r+2\sum_{i>j} m_im_j,
\]
that $\sum_{i>j} m_im_j\le (n^2-r)/2$. Now
$\sum_{i=1}^km_{i\tau}m_i=m_1m_k+\sum_{i=1}^{k-1}m_{i+1}m_i$ and $m_km_1$ is not
a term in the sum as $k>2$. Therefore
\[
  \dim C_F(A,\lambda)\le\sum_{i=1}^km_{i\tau}m_i=m_km_1+\sum_{i=1}^{k-1}m_{i+1}m_i
  \le \sum_{i>j} m_im_j\le \frac{n^2-r}{2}<\frac{n^2}{2}.
\]

The final case to consider is when the derangement $\tau$ has $\ell>1$ cycles.
A typical cycle~$C$ has length $|C|\ge2$, as $\tau$ is a derangement.
Moreover, the reasoning of the previous paragraph gives
$\sum_{j\in C} m_im_{i\tau}\le m_C^2/2$ where
$m_C:=\sum_{j\in C} m_j$.  Therefore
\[
  \dim C_F(A,\lambda)\le\sum_{i=1}^km_{i\tau}m_i
  =\sum_C\sum_{j\in C} m_{i\tau}m_i\le\sum_C \frac{m_C^2}{2}
  <\frac12\left(\sum_C m_C\right)^2=\frac{n^2}{2}.
\]
This completes the proof.
\end{proof}

\begin{remark}
The bound $k_0+m_0/2\le n$ in Theorem~\ref{T:UB} is best possible.
Suppose that $A$ is nilpotent and $n$ is odd. Take
$\mu'_1=(n+1)/2$ and $\mu'_2=(n-1)/2$ in Theorem~\ref{T:NilpDim}.
This implies that $\mu=(2,\dots,2,1)$ with $(n-1)/2$ copies of~$2$.
Furthermore $k_0=\mu'_1=(n+1)/2$ and $m_0=n$. Therefore $k_0+m_0/2=n+1/2$, and
$\dim C_F(A,\lambda)=(\mu'_1)^2+(\mu'_2)^2=(n^2+1)/2\not\le n^2/2$.
\end{remark}

\begin{remark}\label{R:2m}
The bound $n^2/2$ in Theorem~\ref{T:UB} is attained
infinitely often. Suppose $\textup{char}(F)\ne 2$, $n=2m$ is even,
$A=\left(\begin{smallmatrix}I_m&0\\0&-I_m\end{smallmatrix}\right)$,
and $\lambda=-1$. Then $C_F(A,\lambda)$ contains
$\{\left(\begin{smallmatrix}0&R\\ S&0\end{smallmatrix}\right)\mid R,S\in F^{m\times m}\}$. Thus
$\dim C_F(A,\lambda)\ge 2m^2=n^2/2$, and hence $\dim C_F(A,\lambda)=n^2/2$.
\end{remark}

\begin{remark}
The probability $\pi$ that a uniformly random matrix
$A\in\F_q^{n\times n}$ does \emph{not} satisfy $k_0+m_0/2\le n$ turns out to
be small. To show this, first recall that the number of
nilpotent $n\times n$ matrices over $\F_q$ is $q^{n^2-n}$ by \cite{FH}.
Using this fact it is not hard to prove that the number of
$A\in\F_q^{n\times n}$ with $c_A(t)=t^{m_0}h(t)$ and $h(0)\ne0$ is
$q^{n^2-m_0}\omega(n,q)/\omega(m_0,q)$
where $\omega(n,q):=\prod_{i=1}^n(1-q^{-i})$.
Suppose $k_0+m_0/2> n$. Since $m_0\ge k_0$ this implies $3m_0/2>n$ and
$m_0>2n/3$. Let $\pi'$ be the probability that a uniformly random
$A\in\F_q^{n\times n}$ has $c_A(t)=t^{m_0}h(t)$ and $h(0)\ne0$ for some
$m_0>2n/3$. As $\omega(n,q)\le\omega(m_0,q)$, we have
\[
  \pi\le
  \pi'=\frac{1}{q^{n^2}}\sum_{2n/3<m_0\le n}q^{n^2-m_0}\frac{\omega(n,q)}{\omega(m_0,q)}
  <\sum_{2n/3<m_0<\infty}q^{-m_0}=\frac{q^{-\lceil2n/3\rceil}}{1-q^{-1}}.
\]
Hence $\pi< 2q^{-\lceil2n/3\rceil}$ which is \emph{very} small for large $n$.
\end{remark}

Alas~\cite[Theorem~2.1]{CC} is wrong. The proof uses~\cite[Lemma~3]{CG} which
is wrong. If $A$ is an irreducible $2\times 2$ matrix over $F$, then
$C_F(A)=F[A]=\langle I,A\rangle$ is two dimensional (thus achieving that
upper bound $(n-1)^2+1$), however, $A$ does not have the stated form (which
is a reducible matrix). We now correct and extend~\cite[Theorem~2.1]{CC}.
Note that when $A$ is a scalar matrix $C_F(A,\lambda)$ equals $F^{n\times n}$
when $\lambda=1$, and equals $\{0\}$ otherwise.

\begin{proposition}\label{P:UB}
Suppose $A\in F^{n\times n}$ is not a scalar matrix and $0\ne\lambda\in F$. Then
\[
  \dim C_F(A,\lambda)\le (n-1)^2+1.
\]
If $n\ge2$ and equality holds then
$m_A(t)$ is quadratic, and if $n\ge3$ then $m_A(t)$ is reducible.
\end{proposition}

\begin{proof}
The result is trivial if $n=1$. Suppose $n\ge2$.
If $\lambda\ne1$, then Theorem~\ref{T:UB} shows that
$\dim C_F(A,\lambda)\le n^2/2$. This is at most $(n-1)^2+1$.
If $\dim C_F(A,\lambda)=(n-1)^2+1$, then $n=2$ and $\dim C_F(A,\lambda)=2$.
Since $A$ is not a scalar, $m_A(t)$ is not linear. Hence $m_A(t)$ is
quadratic. Henceforth assume that $\lambda=1$.

Let $c_A(t)=f_1(t)^{e_1}\cdots f_r(t)^{e_r}$ where $f_1(t),\dots,f_r(t)$
are $r\ge1$
distinct irreducible polynomials over $F$, and $e_i\ge1$ for each $i$.
View $V=F^n$ as a right $F[A]$-module, or as a module over the principal ideal
ring $F[t]$, as in~\cite{HH}. 
The $f_i$-primary submodule of $V$ is
\[
  V_i:=\{v\in V\mid vf_i(A)^n=0\}.
\]
We have $V=V_1\oplus\cdots\oplus V_r$ and $\dim V_i=d_ie_i$ where
$d_i=\deg(f_i)$. It is well known that $C_F(A)=C_F(A_1)\oplus\cdots\oplus C_F(A_r)$
where $A_i$ acts on $V_i$
and $A=A_1\oplus\cdots\oplus A_r$, see~\cite{HH}. Since
$\dim C_F(A) \le\sum_{i=1}^r (d_ie_i)^2$ and $n=\sum_{i=1}^r d_ie_i$,
we see $\dim C_F(A)\le (n-1)^2+1$ for $r\ge2$. Suppose that equality holds.
Then $r=2$, and without loss of generality, $d_1e_1=n-1$ and $d_2e_2=1$.
Thus $\dim C_F(A_1)=(n-1)^2$ and $\dim C_F(A_2)=1$,
it follows that $A_1=\alpha_1 I_{n-1}$ and $A_2=\alpha_2 I_1$.
Hence $c_A(t)=(t-\alpha_1)^{n-1}(t-\alpha_2)$ and
$m_A(t)=(t-\alpha_1)(t-\alpha_2)$ is a reducible quadratic over $F$
with distinct roots.

Suppose now that $r=1$. If $d_1\ge2$, then $C_F(A)\le K^{n/d_1\times n/d_1}$
for an extension field $K$ of $F$ such that
$\dim_F K=d_1$, and $K\cong F[t]/(f_1(t))$. Therefore
\[
  \dim C_F(A)\le d_1(n/d_1)^2=n^2/d_1\le n^2/2\le (n-1)^2+1.
\]
Suppose now that $d_1=1$, and hence that $c_A(t)=(t-\alpha)^n$ and $A$
is $\alpha$-potent. Since $A-\alpha I$ is nilpotent
and $C_F(A)=C_F(A-\alpha I)$, $A$ is conjugate by an element of $\GL(n,F)$
to a matrix of the form $\alpha_1 I+\bigoplus_{i=1}^{k_0} J_{\mu_i}$ and
$\dim C_F(A)=\sum_{i=1}^{\mu_1} (\mu_i')^2$ by Theorem~\ref{T:NilpDim}.
Since $A\ne \alpha I$ it follows that $\mu_2'=|\{j\mid \mu_j\ge2\}|\ge1$
and hence $\dim C_F(A)\le (n-1)^2+1$.

Suppose now that $r=1$ and $\dim C_F(A,1)= (n-1)^2+1$.
As remarked above, $\dim C_F(A)\le n^2/d_1$. If $d_1\ge2$,
then $n^2/d_1<(n-1)^2+1$ unless $n=d_1=2$. In this case $c_A(t)=m_A(t)$
is an irreducible quadratic. Suppose now that $d_1=1$
and $c_A(t)=(t-\alpha_1)^n$. By hypothesis, $A$ is not a scalar matrix,
so $A\ne\alpha_1I_n$, and hence $\mu'_2=|\{j\mid \mu_j\ge2\}|\ge1$.
Therefore $\sum_{i=1}^{\mu_1} (\mu_i')^2\le (n-1)^2+1$. Suppose that equality holds.
Then
$\mu'_1=n-1$ and $\mu'_2=1$. Hence $\mu_1=2$ and $\mu_i=1$ for $2\le i\le n-1$.
Therefore $\bigoplus_{i=1}^{n-1} J_{\mu_i}$ equals $\alpha_1 I_n+E_{1,2}$
where $E_{1,2}$ is the $n\times n$ matrix with 1 in position
$(1,2)$, and zeros elsewhere. Thus the quadratic
$m_A(t)=(t-\alpha_1)^2$ is a reducible over~$F$.
\end{proof}

\begin{corollary}\label{C:bound}
Suppose $A\in F^{n\times n}$ is not a scalar. Then
$\dim C_F(A,\lambda)\le n^2-n$.
\end{corollary}

\begin{proof}
The result is true if $n=1$. Suppose $n\ge2$.
Since $(n-1)^2+1\le n^2-n$, we may assume that $\lambda=0$
by Proposition~\ref{P:UB}. The case when $\lambda=0$ is handled
by Lemma~\ref{L:k0} as $A\ne0$.
\end{proof}

\section*{Acknowledgments}

Both SPG and CEP gratefully acknowledge the support of the
Australian Research Council Discovery Grant DP160102323.


\begin{thebibliography}{99}

\bib{Aigner}{book}{
   author={Aigner, Martin},
   title={Combinatorial theory},
   series={Classics in Mathematics},
   note={Reprint of the 1979 original},
   publisher={Springer-Verlag, Berlin},
   date={1997},
   pages={viii+483},
   isbn={3-540-61787-6},
   review={\MR{1434477}},
   doi={10.1007/978-3-642-59101-3},
}

\bib{CG}{article}{
   author={Akbari, S.},
   author={Ghandehari, M.},
   author={Hadian, M.},
   author={Mohammadian, A.},
   title={On commuting graphs of semisimple rings},
   journal={Linear Algebra Appl.},
   volume={390},
   date={2004},
   pages={345--355},
   issn={0024-3795},
   review={\MR{2083665}},
   doi={10.1016/j.laa.2004.05.001},
}
	
\bib{CC}{article}{
   author={Alahmadi, Adel},
   author={Alamoudi, Shefa},
   author={Karadeniz, Suat},
   author={Yildiz, Bahattin},
   author={Praeger, Cheryl E.},
   author={Sol{\'e}, Patrick},
   title={Centraliser codes},
   journal={Linear Algebra Appl.},
   volume={463},
   date={2014},
   pages={68--77},
   issn={0024-3795},
   review={\MR{3262389}},
   doi={10.1016/j.laa.2014.08.024},
}
	
\bib{TCC}{article}{
   author={Alahmadi,  Adel},
   author={Glasby, S. P.},
   author={Praeger, Cheryl E.},
   author={Sol\'e, Patrick},
   author={Yildiz, Bahattin},
   title={Twisted Centralizer Codes},
   journal={Linear Algebra Appl. to appear; ArXiv: 1607.04079},
}

\bib{RC}{book}{
    AUTHOR = {Carter, Roger W.},
     TITLE = {Finite groups of {L}ie type: Conjugacy classes and complex characters},
 PUBLISHER = {John Wiley \& Sons, Inc., New York},
      YEAR = {1985},
     PAGES = {xii+544},
      ISBN = {0-471-90554-2},
}

\bib{CM}{book}{
   author={Collingwood, David H.},
   author={McGovern, William M.},
   title={Nilpotent orbits in semisimple Lie algebras},
   series={Van Nostrand Reinhold Mathematics Series},
   publisher={Van Nostrand Reinhold Co., New York},
   date={1993},
   pages={xiv+186},
   isbn={0-534-18834-6},
   review={\MR{1251060}},
}

\bib{FH}{article}{
   author={Fine, N. J.},
   author={Herstein, I. N.},
   title={The probability that a matrix be nilpotent},
   journal={Illinois J. Math.},
   volume={2},
   date={1958},
   pages={499--504},
   issn={0019-2082},
   review={\MR{0096677}},
}

\bib{SG}{article}{
   author={Glasby, S. P.},
   title={On the tensor product of polynomials over a ring},
   journal={J. Aust. Math. Soc.},
   volume={71},
   date={2001},
   number={3},
   pages={307--324},
   issn={1446-7887},
   review={\MR{1862397}},
   doi={10.1017/S1446788700002950},
}

\bib{HH}{book}{
   author={Hartley, B.},
   author={Hawkes, T. O.},
   title={Rings, modules and linear algebra},
   note={A further course in algebra describing the structure of abelian
   groups and canonical forms of matrices through the study of rings and
   modules;
   A reprinting},
   publisher={Chapman \& Hall, London-New York},
   date={1980},
   pages={xi+210},
   isbn={0-412-09810-5},
   review={\MR{619212}},
}

\bib{K}{book}{
   author={Kaplansky, Irving},
   title={Linear algebra and geometry. A second course},
   publisher={Allyn and Bacon, Inc., Boston, Mass.},
   date={1969},
   pages={xii+139},
   review={\MR{0249444}},
}

\bib{NeuP}{article}{
   author={Neumann, Peter M.},
   author={Praeger, Cheryl E.},
   title={Cyclic matrices over finite fields},
   journal={J. London Math. Soc. (2)},
   volume={52},
   date={1995},
   number={2},
   pages={263--284},
   issn={0024-6107},
   review={\MR{1356142}},
   doi={10.1112/jlms/52.2.263},
}

\bib{NP}{article}{
   author={Niemeyer, Alice C.},
   author={Praeger, Cheryl E.},
   title={Estimating proportions of elements in finite groups of Lie type},
   journal={J. Algebra},
   volume={324},
   date={2010},
   number={1},
   pages={122--145},
   issn={0021-8693},
   review={\MR{2646035}},
   doi={10.1016/j.jalgebra.2010.03.018},
}

\end{thebibliography}
\end{document}